\documentclass{article}%
\usepackage{amsmath}
\usepackage{amsfonts}
\usepackage{amssymb}
\usepackage{graphicx}%
\usepackage{mathrsfs}
\providecommand{\U}[1]{\protect \rule{.1in}{.1in}}
\newtheorem{theorem}{Theorem}[section]

\newtheorem{corollary}[theorem]{Corollary}

\newtheorem{definition}[theorem]{Definition}
\newtheorem{example}[theorem]{Example}

\newtheorem{lemma}[theorem]{Lemma}

\newtheorem{proposition}[theorem]{Proposition}
\newtheorem{remark}[theorem]{Remark}

\newenvironment{proof}[1][Proof]{\noindent \textbf{#1.} }{\  \rule{0.5em}{0.5em}}
\begin{document}

\title{Quasi-continuous random variables and processes
under the $G$-expectation framework}
\author{Mingshang Hu \thanks{Zhongtai Institute of Finance, Shandong University,
humingshang@sdu.edu.cn. Research supported by NSF (No. 11201262 and 11301068) and Shandong Province (No. BS2013SF020 and ZR2014AP005)}
\and Falei Wang\thanks{Institute for Advanced Research and School of Mathematics, Shandong University,
flwang2011@gmail.com. Research supported  by the China Postdoctoral Science Foundation (No.2015M582068) and Fundamental Research Funds of Shandong University (No.2015GN023)} \and Guoqiang Zheng\thanks{School of Mathematics, Shandong University,  zhengguoqiang.ori@gmail.com.
Hu, Wang and Zheng's research was
partially supported by NSF (No.10921101) and by the 111
Project (No.B12023)}}
\date{}
\maketitle
\begin{abstract}
In this paper,  we first use PDE techniques and probabilistic methods to identify a kind
of quasi-continuous random variables. Then we  give a
characterization of the $G$-integrable processes and get a kind of
quasi-continuous processes by Krylov's estimates.
 This result is  useful for the development of $G$-stochastic analysis theory.
Moreover, it  also provides a tool for the study of  the non-Markovian It\^{o} processes.
\end{abstract}


\textbf{MSC-classification}: 60H10, 60H30
\section{Introduction}
 Motivated by model uncertainty in
finance, Peng \cite{Peng2004,
Peng2005} firstly constructed a kind of dynamically
consistent fully nonlinear expectations  through PDE approach. An important case is the $G$-expectation
$\mathbb{\hat{E}}[\cdot]$ and  the  corresponding   canonical process
$(B_{t})_{t\geq0}$ is called $G$-Brownian motion analogous to the classical Wiener process.
Under the $G$-expectation framework, the corresponding stochastic
calculus of It\^{o}'s type  was also established in Peng  \cite{P07a, P08a}.

The $G$-expectation can be also seen as a upper expectation. Indeed, Denis et. al. \cite{DHP11} obtained a representation theorem of
$G$-expectation
$\mathbb{\hat{E}}[\cdot]$ by stochastic control method:%
\[
\mathbb{\hat{E}}[X]=\sup_{P\in \mathcal{P}}E_{P}[X]\text{ for }X\in
L_{ip}(\Omega),
\]
where $\mathcal{P}$ is a family of weakly compact probability
measures on $(\Omega,\mathcal{B}(\Omega))$. Moreover, they gave a
characterization of
the space $L_{G}^{p}(\Omega)$ and proved that every random variable in $L_{G}%
^{p}(\Omega)$ is   quasi-continuous.
The representation theorem was also obtained in \cite{HP09} by
a simple probabilistic method.

The present article is devoted to the study of  integrable random variables  and stochastic processes in the $G$-expectation framework.
The classical Lusin's theorem indicates   each random variable is
``quasi-continuous''  in  a probability space. However,
it is difficult to verify a random variable is quasi-continuous in the $G$-expectation framework, since
the measures in $\mathcal{P}$
may be mutually singular.  This problem  has restricted the development of the $G$-stochastic analysis theory.
For example, it is difficult to construct the approximation of an admissible control  to get the dynamic
 programming
 principle
 for $G$-stochastic control problems and we cannot
use the approximation theory of measurable function to prove the Markov property of the $G$-stochastic differential equations.

To overcome this difficult, we use PDE techniques and stochastic control  methods to
obtain some polar sets associated to $X$, which is  a multi-dimensional $G$-It\^{o} process. Based on these polar sets,
we prove some ``irregular'' Borel measurable  functions on $(\Omega, \mathcal{B}(\Omega))$ are quasi-continuous, which implies the
space $L^p_G(\Omega)$ contains enough elements such as
$I_{\{X_{t}\in [a,b]\}}$. Thus the approximation of quasi-continuous random variables    through simple functions is possible.
Indeed, Hu and Ji \cite{HJ2} studied the $G$-stochastic control problems with the help of this result.
In $1$-dimensional case, Martini \cite{MC} also got some polar sets by a pure probabilistic approach. By our arguments, we also obtain the convergence rate,
which enables us to study the
sample path properties of the non-Markovian It\^{o} processes, such as the differentiability  and the maxima.

The similar questions  arise for the $G$-integrable processes, and the rest of this paper is devoted to studying the space $M_G^p(0,T)$.
First, we give a characterization of $M_{G}^{p}(0,T)$, which non-trivially generalizes the result of \cite{DHP11}. Moreover,
we establish a monotone convergence theorem for quasi-continuous processes.
Next we  apply Krylov's estimates to get a kind of quasi-continuous
processes. In particular, these estimates induce a  weak dominated convergence theorem for  $G$-It\^{o} processes, which is useful for
the study of $G$-stochastic analysis. For example,   this result can be used  to deal with the well-posedness of $G$-backward stochastic differential equations under non-Lipschitz
condition.

This paper is organized as follows. In section 2, we recall some
necessary notations and results of $G$-expectation theory. In
section 3, we study the polar sets and give some useful
quasi-continuous random variables. In section 4, we obtain the
characterization of $M_{G}^{p}(0,T)$ and get some useful quasi-continuous progressively measurable processes by Krylov's estimates.

\section{Preliminaries}

The main purpose of this section is to recall some basic notions and
results of $G$-expectation, which are needed in the sequel. The
readers may refer to \cite{P07a}, \cite{P08a}, \cite{P10}, \cite{PengICM2010} for more
details.

Let $\Omega=C_{0}^{d}(\mathbb{R}^{+})$ be the space of all $\mathbb{R}^{d}%
$-valued continuous paths $(\omega_{t})_{t\geq0}$, with
$\omega_{0}=0$,
equipped with the distance%
\[
\rho(\omega^{1},\omega^{2}):=\sum_{i=1}^{\infty}2^{-i}[(\max_{t\in
\lbrack 0,i]}|\omega_{t}^{1}-\omega_{t}^{2}|)\wedge1].
\]
For each $t\in \lbrack0,\infty)$, we denote
\begin{itemize}
\item $B_{t}(\omega):=\omega_{t}$ for each $\omega \in \Omega$;
\item $\mathcal{B}(\Omega)$: the Borel $\sigma$-algebra of $\Omega$,\ $\Omega_{t}:=\{ \omega_{\cdot \wedge t}:\omega \in \Omega \}$,\
$\mathcal{F}_{t}:=\mathcal{B}(\Omega_{t})$;
\item $L^{0}(\Omega)$: the space of all $\mathcal{B}(\Omega)$-measurable real
functions;

\item $L^{0}(\Omega_{t})$: the space of all $\mathcal{B}(\Omega_{t}%
)$-measurable real functions;

\item $B_{b}(\Omega)$: all bounded elements in $L^{0}(\Omega)$; $B_{b}%
(\Omega_{t}):=B_{b}(\Omega)\cap L^{0}(\Omega_{t})$;

\item $C_{b}(\Omega)$: all continuous elements in $B_{b}(\Omega)$;
$C_{b}(\Omega_{t}):=C_{b}(\Omega)\cap L^{0}(\Omega_{t})$;
\item $L_{ip}(\Omega):=\{ \varphi(B_{t_{1}},\ldots,B_{t_{k}}):k\in \mathbb{N}%
,t_{1},\ldots,t_{k}\in \lbrack0,\infty),\varphi \in C_{b.Lip}(\mathbb{R}%
^{k\times d })\}$, where $C_{b.Lip}(\mathbb{R}^{k\times d})$ denotes the space of
bounded and
Lipschitz functions on $\mathbb{R}^{k\times d}$; $L_{ip}(\Omega_{t}):=L_{ip}%
(\Omega)\cap L^{0}(\Omega_{t})$.
\end{itemize}

For each given monotonic and sublinear function $G:\mathbb{S}%
(d)\rightarrow \mathbb{R}$, let the canonical process
$B_{t}=(B_{t}^{i})_{i=1}^{d}$ be the $d$-dimensional $G$-Brownian
motion in the $G$-expectation space $(\Omega,L_{ip}(\Omega),\mathbb{\hat{E}%
}[\cdot],(\mathbb{\hat{E}}_{t}[\cdot])_{t\geq0})$, where $\mathbb{S}(d)$ denotes the space of all $d\times d$ symmetric matrices. For each
$p\geq1$, the completion of $L_{ip}(\Omega)$ under the norm
$||X||_{L_{G}^{p}}:=(\mathbb{\hat{E}}[|X|^{p}])^{1/p}$ is denoted by
$L_{G}^{p}(\Omega)$. Similarly, we can define
$L_{G}^{p}(\Omega_{T})$ for each fixed $T\geq0$. In this paper, we
always assume that $G$ is non-degenerate, i.e., there exist two constants
$0<\underline
{\sigma}^{2}\leq \bar{\sigma}^{2}<\infty$ such that
\[
\frac{1}{2}\underline{\sigma}^{2}\text{tr}[A-B]\leq G(A)-G(B)\leq \frac{1}%
{2}\bar{\sigma}^{2}\text{tr}[A-B], \text{ for }A\geq B.
\]
Then we  deduce that $|G(A)|\leq
\frac{1}{2}\bar{\sigma}^{2}\sqrt {d}\sqrt{\text{tr}[AA^{T}]}$ for
any $A\in \mathbb{S}(d)$.

Denis et. al. \cite{DHP11} proved that the completions of
$C_{b}(\Omega)$ and $L_{ip} (\Omega)$ under
$\Vert\cdot\Vert_{L_{G}^{p}}$ are the same.
\begin{theorem}[\cite{DHP11,HP09}]
\label{the2.7}  There exists a weakly compact  set of probability
measures $\mathcal{P}$ on $(\Omega,\mathcal{B}(\Omega))$, such that
\[
\mathbb{\hat{E}}[\xi]=\sup_{P\in\mathcal{P}}E_{P}[\xi],\  \text{for
 all}\ \xi\in  {L}_{G}^{1}{(\Omega)}.
\]
$\mathcal{P}$ is called a set that represents $\mathbb{\hat{E}}$.
\end{theorem}
\begin{remark}\label{yq1}{\upshape Denis et. al. \cite{DHP11} constructed a concrete set $\mathcal{P}_M$ that represents $\hat{\mathbb{E}}$.
For  simplicity's sake, we consider  the $1$-dimensional case,  thus  $G(a)=\frac{1}{2}(\bar{\sigma}^2a^+-\underline{\sigma}^2a^-)$ for each $a\in\mathbb{R}$.
Suppose   $B$ is a  Brownian motion defined  on $(\Omega,L^0(\Omega),P)$, then
\[
\mathcal{P}_M := \{P_{\theta} : P_{\theta}= P\circ X^{-1},\ X_t = \int^t_0 \theta_sdB_s,\  \theta\in L^2_{\mathcal{F}}([0, T ]; [\underline{\sigma}^2, \bar{\sigma}^2])\}\]
 represents $\hat{\mathbb{E}}$, where $L^2_{\mathcal{F}}([0, T ]; [\underline{\sigma}^2, \bar{\sigma}^2])$ is the collection of all adapted
measurable processes with $\underline{\sigma}^2 \leq |\theta_s|^2 \leq\bar{\sigma}^2$.
}
\end{remark}

Let $\mathcal{P}$ be a weakly compact set that represents
$\mathbb{\hat{E}}$.
For this $\mathcal{P}$, we define capacity%
\[
c(A):=\sup_{P\in\mathcal{P}}P(A),\ A\in\mathcal{B}(\Omega).
\]
An important property of this capacity is that $c(F_{n})\downarrow c(F)$ for
any closed sets $F_{n}\downarrow F$.

A set $A\subset\mathcal{B}(\Omega)$ is polar if $c(A)=0$.  A
property holds ``$quasi$-$surely$'' (q.s.) if it holds outside a
polar set. In the following, we do not distinguish  between random
variables $X$ and $Y$ if $X=Y$ q.s..

\begin{definition}A real function $X$ on $\Omega$ is said to be quasi-continuous if for each $\varepsilon>0$,
there exists an open set $O$ with $c(O)<\varepsilon$ such that
$X|_{O^{c}}$ is continuous.
\end{definition}

\begin{definition}  We say that  $X:\Omega\mapsto\mathbb{R}$ has a quasi-continuous
version if there exists a quasi-continuous function
$Y:\Omega\mapsto\mathbb{R}$ such that $X = Y$, q.s..
\end{definition}

\begin{theorem}[\cite{DHP11,HP09}]
\label{new-qua}We have
\begin{align*}
L_{G}^{p}(\Omega)=\{X\in L^0(\Omega)\ :\ \
&\lim\limits_{N\rightarrow\infty}\mathbb{\hat{E}}[|X|^pI_{|X|\geq
N}]=0 \
 \text{and}\\ &  X\ \text {has a quasi-continuous version}\}.
 \end{align*}
\end{theorem}

\begin{theorem}[\cite{DHP11,HP09}]
\label{new-qua1} Let $(X_{k})_{k\geq 1}\subset L_{G}^{1}(\Omega)$, be such
that $X_{k}\downarrow X$ q.s.. Then
$\mathbb{\hat{E}}[X_{k}]\downarrow \mathbb{\hat{E}}[X]$. In
particular, if $X\in L_{G}^{1}(\Omega)$, then $\mathbb{\hat{E}}[|X_{k}%
-X|]\downarrow0$.
\end{theorem}

\begin{definition}[\cite{HJY}]
\label{esssup}Assume $X_{\theta}\in L^{1}_G(\Omega_t)$ for each $\theta\in\Theta$. Then the essential supremum of $\{X_{\theta}\mid
\theta\in\Theta\}$, denoted by $\underset{\theta\in\Theta}{\text{ess}\sup} X_{\theta}$, is a random variable $\zeta\in L_{G}^{1}(\Omega_{t})$ satisfying:
\begin{description}
\item[(i)] $\forall  \theta\in\Theta,$ $\zeta\geq X_{\theta}$ $\ $q.s.;

\item[(ii)] if $\xi$ is a random variable satisfying $\xi\geq X_{\theta}$
$\ $q.s. for any $\theta\in\Theta$, then $\zeta\leq\xi$ $\ $q.s..\end{description}
\end{definition}

\begin{definition}
\label{def2.6} Let $M_{G}^{0}(0,T)$ be the collection of processes
of  the following form: for a given partition
$\{t_{0},\cdot\cdot\cdot,t_{N}\}=\pi _{T}$ of $[0,T]$,
\[
\eta_{t}(\omega)=\sum_{i=0}^{N-1}\xi_{i}(\omega)I_{[t_{i},t_{i+1})}(t),
\]
where $\xi_{i}\in L_{ip}(\Omega_{t_{i}})$,
$i=0,1,2,\cdot\cdot\cdot,N-1$. For each $p\geq1$,  denote by
$M_{G}^{p}(0,T)$ the completion of $M_{G}^{0}(0,T)$ under the norm
$||\eta||_{M_{G}^{p}}:=(\mathbb{\hat{E}}[\int_{0}^{T}|\eta_{t}|^{p}dt])^{1/p}$.
\end{definition}

For each $ \eta\in M_{G}^{2}(0,T)$, the $G$-It\^{o} integral
$\{\int^t_0\eta_sdB^i_s\}_{t\in[0,T]}$ is well defined, see Peng
\cite{P10} and Li-Peng \cite{L-P}.

\section{Quasi-continuous random variables}

In this section, we shall   prove some ``irregular''
Borel measurable functions on $\Omega$ are quasi-continuous  by
virtue of a PDE approach. We consider the following
$G$-It\^{o} processes (in this paper we always use Einstein's summation
convention): for each given $x=(x_{1},\ldots,x_{n})^{\top}\in
\mathbb{R}^{n}$ and $1\leq i\leq n$,
\[
X_{t}^{x_{i};i}=x_{i}+\int_{0}^{t}\alpha_{i}(s)ds+\int_{0}^{t}\beta_{i}%
^{jk}(s)d\langle B^{j},B^{k}\rangle_{s}+\int_{0}^{t}\sigma_i(s)dB_{s},
\]
where $\beta^{jk}(t)=\beta^{kj}(t)$ and $\sigma_i$ is the $i$-th row of $\sigma$. Denote by  $X_{t}^{x}=(X_{t}%
^{x_{1};1},\ldots,X_{t}^{x_{n};n})^{\top}$, $\alpha(t)=(\alpha_{1}(t),\ldots,\alpha_{n}(t))^{\top}$ and
$\beta^{jk}(t)=(\beta_{1}^{jk}(t),\ldots,\beta_{n}^{jk}(t))^{\top}$. Then the above $G$-It\^{o} processes can
be written as%
\begin{equation}
X_{t}^{x}=x+\int_{0}^{t}\alpha_sds+\int_{0}^{t}\beta^{jk}_sd\langle B^{j},B^{k}\rangle_{s}+\int_{0}^{t}\sigma_sdB_{s}.
\label{SDE1}%
\end{equation}

In this paper, we shall use the following assumptions:
\begin{description}
\item[(H1)] For each $s>0$, $(\alpha_t)_{0\leq t\leq s}$ and $(\beta_t^{jk})_{0\leq t\leq s}$ are in $M^2_G(0,s;\mathbb{R}^n)$, $(\sigma_t)_{0\leq t\leq s}$ are in $M^2_G(0,s;\mathbb{R}^{n\times d})$;
\item[(H2)] There exists a constant $L>0$ such that for each $t\in
\lbrack0,\infty)$,%
\[
|\alpha_{i}(t)|\leq L\text{, }|\beta_{i}^{jk}(t)|\leq L,|\sigma_{i}(t)|\leq L, \text{ for }j,k\leq d\ \text{and}\ i\leq
n;
\]
\item[(H3)] There exist  two constants $0<\lambda< \Lambda<\infty$ such that for
each $t\in \lbrack0,\infty)$, %
\begin{align*}
&\lambda I_{n\times n}  \leq \sigma_t(\sigma_t)^{\top} \leq \Lambda
I_{n\times n}, \text{ if }n\leq d, \\
&\lambda I_{d\times d}\leq (\sigma_t)^{\top}\sigma_t \leq \Lambda
I_{d\times d}, \text{ if }n>d;
\end{align*}

\item[(H4)] There exist two constants $0<\gamma<\Gamma<\infty$ such that for
each $(t,x)\in \lbrack0,\infty)\times \mathbb{R}^{n}$,%
\[
\gamma \leq|\sigma_{i}(t)|^{2}=\sigma_{i}(t)(\sigma_{i}(t))^{\top}\leq
\Gamma, \text{ for } \ i\leq n.
\]

\end{description}

\begin{remark}{\upshape
If $n\leq d$, then (H3) is stronger than (H4).}
\end{remark}

In order to state the main results of this section, we shall use the stochastic representation for the  HJB equation.
For this purpose, we denote the following sets:  $$
\mathcal{V}=\{v=(\alpha,\beta,\sigma)| \ \text{$\alpha, \beta $ and $\sigma$ satisfy assumptions (H1), (H2) and (H3)}  \}
$$ and
  $$
\mathcal{V}_0=\{v\in \mathcal{V}|  \  \text{ $v$ is a constant process, i.e., $v(t)=v(0)$ for each $t> 0$ }\}.
$$
Now for each fixed $t\geq0$, $v\in\mathcal{V}$ and for each given $\xi \in L_{G}^{2}(\Omega
_{t};\mathbb{R}^{n})$, consider the following $G$-It\^{o} process:%
\begin{equation}
X_{s}^{t,\xi,v}=\xi+\int_{t}^{s}\alpha_rdr+\int_{t}^{s}\beta^{jk}_rd\langle B^{j},B^{k}\rangle_{r}+\int_{t}^{s}\sigma_rdB_{r}. \label{SDE2}%
\end{equation}
 Then for each fixed  $T>0$ and $\Phi \in C_{b.Lip}(\mathbb{R}^{n})$, we define%
\[
Y_{t}^{t,\xi}=\underset{v\in\mathcal{V}}{ess\sup}\mathbb{\hat{E}}_{t}[\Phi(X_{T}^{t,\xi,v})], \ \ t\in[0,T].
\]
Next, for each $x\in \mathbb{R}^{n}$, we set%
\[
u(t,x):=Y_{t}^{t,x}.
\]
It is important to note that
$u(0,x)=\underset{v\in\mathcal{V}}{\sup}\mathbb{\hat{E}}[\Phi(X_{T}^{0,x,v})]$.

\begin{theorem}[\cite{HJY}]
\label{th0}For each fixed  $T>0$, we have

\begin{description}
\item[(1)] $u(t,x)$ is a deterministic continuous function of $(t,x)$;

\item[(2)] For each $\xi \in L_{G}^{2}(\Omega_{t};\mathbb{R}^{n})$,
$Y_{t}^{t,\xi}=u(t,\xi)$;

\item[(3)] $u$ is the unique viscosity solution of the following PDE:%
\begin{equation}
\left \{
\begin{array}
[c]{l}%
\partial_{t}u+\sup\limits_{v\in\mathcal{V}_0}\{G(\sigma^{\top}D_{x}^{2}u\sigma+2\langle \beta^{jk},D_{x}u\rangle)+\langle \alpha,D_{x}u\rangle\}=0,\ (t,x)\in[0,T)\times\mathbb{R}^n,\\
u(T,x)=\Phi(x).
\end{array}
\right.  \label{PDE}%
\end{equation}
\end{description}
\end{theorem}
\begin{remark}\label{myq14}
{\upshape In the definition of $\mathcal{V}$, we can also assume  $\alpha, \beta $ and $\sigma$ satisfy assumptions (H1), (H2) and (H4). In this case,
the set $\mathcal{V}_0$, value function $u$ and equation \eqref{PDE} have to be redefined accordingly.
}
\end{remark}
\subsection{Polar sets associated to $G$-It\^{o} processes }
In the following, we first prove that $c(\{X_{t}^{x}=a\})=0$ for any
$t>0$ and $a\in \mathbb{R}^{n}$, i.e.  the $G$-It\^{o} process $X^x_t$ does not weight single point.
The proof is based on an estimate of the solution to PDE \eqref{PDE}.
\begin{lemma}
\label{le1}Let $T>0$, $\rho=(n\wedge
d)\lambda \underline{\sigma}^{2}(8d\bar{\sigma}^{2}\Lambda)^{-1}$,
$\theta=(2d\bar{\sigma}^{2}\Lambda)^{-1}$, $\varepsilon=(8\kappa)^{-1}\wedge T$,
$m\geq8\kappa$ and $u_{m}$ be the solution of PDE \eqref{PDE} with the terminal
condition $u_{m}(T,x)=\exp(-\frac{m\theta|x-a|^{2}}{2})$, where $a=(a_{1}%
,\ldots.a_{n})^{\top}\in \mathbb{R}^{n}$, $n\wedge d=\min \{n,d\}$,%
\[
\kappa=L^{2}(\bar{\sigma}^{2}%
d\sqrt{d}+1)^{2}((n\wedge d)\lambda \underline{\sigma}^{2})^{-1}.
\]%
Then for any $(t,x)\in \lbrack T-\varepsilon,T)\times \mathbb{R}^{n}$, we have%
\begin{equation}
0\leq u_{m}(t,x)\leq(1+m(T-t))^{-\rho}.\label{So}%
\end{equation}

\end{lemma}

\begin{proof}
It is easy to check that $\bar{u}_{m}(t,x)=0$ is a viscosity subsolution of PDE
(\ref{PDE}). Thus by comparison theorem we get $u_{m}(t,x)\geq0$ for each
$(t,x)\in \lbrack0,T]\times \mathbb{R}^{n}$. Set
\begin{equation}
\tilde{u}_{m}(t,x)=(1+m(T-t))^{-\rho}\exp(-\frac{m\theta|x-a|^{2}%
}{2(1+m(T-t))}).\label{Sup}%
\end{equation}
It is obvious that $\tilde{u}_{m}(T,x)=\exp(-\frac{m\theta|x-a|^{2}}{2})$. In
the following, we shall show that $\tilde{u}_{m}$ is a viscosity supersolution of
PDE (\ref{PDE}) if $t\geq T-\varepsilon$.
It is easy to verify that%
\[
\partial_{t}\tilde{u}_{m}=\frac{\rho m}{1+m(T-t)}\tilde{u}_{m}-\frac
{m^{2}\theta|x-a|^{2}}{2(1+m(T-t))^{2}}\tilde{u}_{m},
\]%
\[
\partial_{x_{i}}\tilde{u}_{m}=-\frac{m\theta(x_{i}-a_{i})}{1+m(T-t)}\tilde
{u}_{m},
\]%
\[
\partial_{x_{i}x_{i}}^{2}\tilde{u}_{m}=-\frac{m\theta}{1+m(T-t)}\tilde{u}%
_{m}+\frac{m^{2}\theta^{2}|x_{i}-a_{i}|^{2}}{(1+m(T-t))^{2}}\tilde{u}_{m},
\]%
\[
\partial_{x_{i}x_{j}}^{2}\tilde{u}_{m}=\frac{m^{2}\theta^{2}(x_{i}-a_{i}%
)(x_{j}-a_{j})}{(1+m(T-t))^{2}}\tilde{u}_{m},\text{ }i\neq j.
\]%
For each $v\in\mathcal{V}_0$, by the assumptions (H1)-(H3), we obtain that
\begin{align*}
&G(-\sigma^{\top}\sigma)\leq-\frac{\underline{\sigma}^{2}}{2}\text{tr}[\sigma
^{\top}\sigma]\leq-\frac{1}{2}(n\wedge d)\lambda \underline{\sigma}^{2},\\&
G(\sigma^{\top}(x-a)(x-a)^{T}\sigma)\leq \frac{\bar{\sigma}^{2}}{2}|x-a|^{2}%
\text{tr}[\sigma^{\top}\sigma]\leq \frac{1}{2}d\Lambda \bar{\sigma}^{2}|x-a|^{2},
\\
&G((-\langle \beta^{jk},x-a\rangle)_{j,k=1}^{d})
\leq \frac{1}{2}L\bar{\sigma
}^{2}d\sqrt{d}|x-a|,\ \
-\langle \alpha,x-a\rangle\leq L|x-a|.
\end{align*}
Note that $L(\bar{\sigma
}^{2}d\sqrt{d}+1)|x-a|\leq L^2(\bar{\sigma
}^{2}d\sqrt{d}+1)^2|x-a|^2((n\wedge d)\lambda \underline{\sigma}^{2})^{-1}+\frac{1}{4}(n\wedge d)\lambda \underline{\sigma}^{2}$.
Then for $(t,x)\in \lbrack T-\varepsilon,T)\times \mathbb{R}^{n}$,
 we have%
\begin{align*}
& \partial_{t}\tilde{u}_{m}+\sup\limits_{v\in\mathcal{V}_0}\{G(\sigma^{\top}D_{x}^{2}\tilde{u}_{m}\sigma+(2\langle
\beta^{jk}(t,x),D_{x}\tilde{u}_{m}\rangle)_{j,k=1}^{d})+\langle \alpha,D_{x}\tilde
{u}_{m}\rangle \}\\
& \leq \partial_{t}\tilde{u}_{m}+\frac{m\theta \tilde{u}_{m}}{1+m(T-t)}%
\sup\limits_{v\in\mathcal{V}_0}G(-\sigma^{\top}\sigma)+\frac{m^{2}\theta^{2}\tilde{u}_{m}}{(1+m(T-t))^{2}%
}\sup\limits_{v\in\mathcal{V}_0}G(\sigma^{\top}(x-a)(x-a)^{\top}\sigma)\\
& +\frac{2m\theta \tilde{u}_{m}}{1+m(T-t)}\sup\limits_{v\in\mathcal{V}_0}G((-\langle \beta^{jk},x-a\rangle
)_{j,k=1}^{d})+\frac{m\theta \tilde{u}_{m}}{1+m(T-t)}\sup\limits_{v\in\mathcal{V}_0}\{-\langle \alpha,x-a\rangle\}
\\
& \leq-\frac{m\theta \tilde{u}_{m}}{1+m(T-t)}|x-a|^{2}(\frac{m}{4(1+m(T-t))}%
-\kappa)\\
& \leq-\frac{m\theta \tilde{u}_{m}}{1+m(T-t)}|x-a|^{2}(\frac{m}{4(1+m\varepsilon
)}-\kappa)\\
& =-\frac{m\theta \tilde{u}_{m}}{1+m(T-t)}|x-a|^{2}\times \frac{m-8\kappa
}{8(1+m\varepsilon)}\\
& \leq0,
\end{align*}
which implies that $\tilde{u}_{m}$ is a viscosity supersolution of PDE
(\ref{PDE}) if $t\geq T-\varepsilon$. Thus by comparison theorem we obtain for
$(t,x)\in \lbrack T-\varepsilon,T]\times \mathbb{R}^{n}$,
\[
u_{m}(t,x)\leq \tilde{u}_{m}(t,x)\leq(1+m(T-t))^{-\rho}.
\]
The proof is complete.
\end{proof}
\begin{remark}\label{myq10}
{\upshape
If $\alpha=\beta^{jk}=0$. From the above proof, we can take $\rho=(n\wedge
d)\lambda \underline{\sigma}^{2}(2d\bar{\sigma}^{2}\Lambda)^{-1}$,
$\theta=(d\bar{\sigma}^{2}\Lambda)^{-1}$, $\varepsilon= T$ ($\kappa=0$),
$m\geq 0$ and the results also hold true.
}
\end{remark}

\begin{remark}
{\upshape
We remark that there is a potential to extend our results to a much more general nonlinear expectation
setting.  In particular, by slightly
more involved estimates, our results still hold for the following PDE (see \cite{HJPS, HJPS1, HJY}):
\begin{equation*}
\left \{
\begin{array}
[c]{l}%
\partial_{t}u+\sup\limits_{v\in\mathcal{V}_0}\{G(\sigma^{\top}D_{x}^{2}u\sigma+2\langle \beta^{jk},D_{x}u\rangle)+f_1(t,D_xu,v))+\langle
\alpha,D_{x}u\rangle+f_2(t,D_xu,v)\}\\ \  =0,\\
u(T,x)=\Phi(x),
\end{array}
\right.
\end{equation*}
where $f_i$ ($i=1,2$) is a Lipschitz continuous function satisfying
$f_i(t,0,v)=0$. The proof is the same without any difficulty. }
\end{remark}

\begin{theorem}
\label{th2}Assume   $\emph{(H1), (H2)}$ and  $\emph{(H3)}$ hold. Then
we have for each $T>0$
\begin{equation}
\mathbb{\hat{E}}[\exp(-\frac{m\theta|X_{T}^{x}-a|^{2}}{2})]\leq(1+m(T\wedge
\varepsilon))^{-\rho},\label{e1}%
\end{equation}
where $X^x_t$ is the $G$-It\^{o} process \eqref{SDE1} and $\theta, \rho, \varepsilon $ are given in Lemma \emph{\ref{le1}}.
In particular, we have%
\begin{equation}
c(\{X_{T}^{x}=a\})=0.\label{e2}%
\end{equation}

\end{theorem}

\begin{proof}
If $T\leq \varepsilon$, it follows from Lemma \ref{le1} and $\mathbb{\hat{E}}[\exp
(-\frac{m\theta|X_{T}^{x}-a|^{2}}{2})]\leq u_{m}(0,x)$ that $\mathbb{\hat
{E}}[\exp(-\frac{m\theta|X_{T}^{x}-a|^{2}}{2})]\leq(1+mT)^{-\rho}$. If
$T>\varepsilon$, by Theorem \ref{th0}(2) and Lemma \ref{le1},  we get that%
\begin{align*}
\mathbb{\hat{E}}[\exp(-\frac{m\theta|X_{T}^{x}-a|^{2}}{2})]  & =\mathbb{\hat
{E}}[\mathbb{\hat{E}}_{T-\varepsilon}[\exp(-\frac{m\theta|X_{T}^{T-\varepsilon
,X_{T-\varepsilon}^{x}}-a|^{2}}{2})]]\\
& \leq \mathbb{\hat{E}}[u_{m}(T-\varepsilon,X_{T-\varepsilon}^{x})]\\
& \leq \mathbb{\hat{E}}[(1+m\varepsilon)^{-\rho}]\\
& =(1+m\varepsilon)^{-\rho}.
\end{align*}
Thus we obtain equation (\ref{e1}). Note that $\exp(-\frac{m\theta|X_{T}^{x}-a|^{2}}%
{2})\geq I_{\{X_{T}^{x}=a\}}$, then%
\[
c(\{X_{T}^{x}=a\})\leq \mathbb{\hat{E}}[\exp(-\frac{m\theta|X_{T}^{x}-a|^{2}}%
{2})]\leq(1+m(T\wedge \varepsilon))^{-\rho}.
\]
Thus we can get $c(\{X_{T}^{x}=a\})=0$ by letting $m\rightarrow \infty$.
\end{proof}

We remark that Martini \cite{MC} proved a similar result in the one dimensional case. By a probabilistic method,
he obtained that the It\^{o} process does not weight single point  under strict ellipticity condition.
In the Theorem \ref{th2}, we also obtain the convergence rate \eqref{e1}, which can be used to estimate the quality of the $G$-It\^{o} processes staying in a ball.

\begin{corollary}\label{myq12}
Assume  $\emph{(H1), (H2)}$ and  $\emph{(H3)}$ hold and $\alpha=\beta^{jk}=0$. Then  for each $t>0$, $y\in\mathbb{R}^n$ and $\epsilon>0$, we have
\[
c(\{|X_t^x-y|\leq \epsilon\})\leq \exp(\frac{\theta}{2})\frac{\epsilon^{2\rho}}{t^{\rho}},
\]
where
$\rho=(n\wedge
d)\lambda \underline{\sigma}^{2}(2d\bar{\sigma}^{2}\Lambda)^{-1}$, $\theta=(d\bar{\sigma}^{2}\Lambda)^{-1}$.
In particular,
\[
\lim\limits_{\epsilon\downarrow 0} \sup\limits_{y\in\mathbb{R}^d}c(\{|X^x_t-y|\leq \epsilon\})=0.
\]
\end{corollary}
\begin{proof}
By Remark \ref{myq10} and Theorem \ref{th2}, we obtain for each $y\in\mathbb{R}^n$ and $m\geq 0$,
\[
\mathbb{\hat{E}}[\exp(-\frac{m\theta|X^x_t-y|^{2}}{2})]\leq \frac{1}{(1+mt)^{\rho}}.
\]
Thus we get for each $m$ and $\epsilon>0$,
\[
\hat{\mathbb{E}}[I_{\{|X^x_t-y|\leq \epsilon\}}]
\leq \exp(\frac{m\theta\epsilon^2}{2})\mathbb{\hat{E}}[\exp(-\frac{m\theta|X^x_t-y|^{2}}{2})]\leq \exp(\frac{m\theta\epsilon^2}{2})\frac{1}{(1+mt)^{\rho}}.
\]
In particular, taking $m=\frac{1}{\epsilon^2}$, we get for each $y\in\mathbb{R}^n$,
\[
c(\{|X^x_t-y|\leq \epsilon\})\leq \exp(\frac{\theta}{2})\frac{\epsilon^{2\rho}}{t^{\rho}},
\]
which completes the proof.
\end{proof}

\begin{example}{\upshape
From the  Corollary \ref{myq12}, we can obtain that for each $t>0$, $y\in\mathbb{R}^d$ and $\epsilon>0$,
\[
c(\{|B_t-y|\leq \epsilon\})\leq \exp(\frac{\theta}{2})\frac{\epsilon^{2\rho}}{t^{\rho}},
\]
where $\rho=\frac{\underline{\sigma}^2}{2\bar{\sigma}^2}$, $\theta=(d\bar{\sigma}^{2})^{-1}$.
This inequality provides a way to study the sample path properties of  non-Markovian It\^{o} process in the Wiener space.
Indeed by Remark \ref{yq1}, we
have \[
P(\{|X_t-y|\leq \epsilon\})\leq \exp(\frac{\theta}{2})\frac{\epsilon^{2\rho}}{t^{\rho}}
\]
and $X_t = \int^t_0 \theta_sdB_s$  is non-differentiable almost everywhere(see \cite{WZ}).
}
\end{example}

By Remark \ref{myq14}, we conclude also  the value function $u$ is the viscosity solution of PDE \eqref{PDE} under the assumptions (H1), (H2) and (H4).
Then we have the following result.
\begin{lemma}
\label{le3} Let $T>0$, $\rho
=\gamma \underline{\sigma}^{2}(8\bar{\sigma}^{2}\Gamma)^{-1}$, $\theta
=(2\bar{\sigma}^{2}\Gamma)^{-1}$, $\varepsilon=(8\kappa)^{-1}\wedge T$, $m\geq8\kappa$
and $u_{m}$ be the solution of PDE \eqref{PDE} with terminal condition
$u_{m}(T,x)=\exp(-\frac{m\theta|x_{i}-a_{i}|^{2}}{2})$, where $a_{i}%
\in \mathbb{R}$,  $\kappa=L^{2}(\bar{\sigma}^{2}d\sqrt{d}%
+1)^{2}(\gamma \underline{\sigma}^{2})^{-1}$. Then for any $(t,x)\in \lbrack
T-\varepsilon,T)\times \mathbb{R}^{n}$, we have%
\begin{equation}
0\leq u_{m}(t,x)\leq(1+m(T-t))^{-\rho}.\label{So2}%
\end{equation}
\end{lemma}

\begin{proof}
The proof of $u_{m}(t,x)\geq0$ is the same as in Lemma \ref{le1}. Set%
\begin{equation}
\tilde{u}_{m}(t,x)=(1+m(T-t))^{-\rho}\exp(-\frac{m\theta|x_{i}-a_{i}|^{2}%
}{2(1+m(T-t))}).\label{Sup1}%
\end{equation}
It is obvious that $\tilde{u}_{m}(T,x)=\exp(-\frac{m\theta|x_{i}-a_{i}|^{2}}%
{2})$. In the following, we show that $\tilde{u}_{m}$ is a viscosity
supersolution of PDE (\ref{PDE}) if $t\geq T-\varepsilon$. It is easy to
verify that, for each $v\in\mathcal{V}_0$%
\[
\partial_{t}\tilde{u}_{m}=\frac{\rho m}{1+m(T-t)}\tilde{u}_{m}-\frac
{m^{2}\theta|x_{i}-a_{i}|^{2}}{2(1+m(T-t))^{2}}\tilde{u}_{m},
\]%
\[
\partial_{x_{i}}\tilde{u}_{m}=-\frac{m\theta(x_{i}-a_{i})}{1+m(T-t)}\tilde
{u}_{m},\text{ }%
\]%
\[
\partial_{x_{i}x_{i}}^{2}\tilde{u}_{m}=-\frac{m\theta}{1+m(T-t)}\tilde{u}%
_{m}+\frac{m^{2}\theta^{2}|x_{i}-a_{i}|^{2}}{(1+m(T-t))^{2}}\tilde{u}_{m},
\]%
\[
\partial_{x_{j}}\tilde{u}_{m}=0,\text{ }\partial_{x_{i}x_{j}}^{2}\tilde{u}%
_{m}=0,\text{ }j\neq i,
\]%
\[
\sigma^{\top}D_{x}^{2}\tilde{u}_{m}\sigma=(\partial_{x_{i}x_{i}}^{2}\tilde{u}%
_{m})\sigma_{i}^{\top}\sigma_{i},
\]%
\[
G(-\sigma_{i}^{\top}\sigma_{i})\leq-\frac{\gamma \underline{\sigma}^{2}}{2};\text{
}G(\sigma_{i}^{\top}\sigma_{i})\leq \frac{\bar{\sigma}^{2}\Gamma}{2},
\]%
\[
(\langle \beta^{jk},D_{x}\tilde{u}_{m}\rangle)_{j,k=1}^{d}=(\partial_{x_{i}%
}\tilde{u}_{m})(\beta_{i}^{jk})_{j,k=1}^{d}.
\]
Then for each  $(t,x)\in \lbrack T-\varepsilon,T)\times \mathbb{R}^{n}$, we have%
\begin{align*}
& \partial_{t}\tilde{u}_{m}+\sup\limits_{v\in\mathcal{V}_0}\{G(\sigma^{\top}D_{x}^{2}\tilde{u}_{m}\sigma+(2\langle
\beta^{jk},D_{x}\tilde{u}_{m}\rangle)_{j,k=1}^{d})+\langle \alpha,D_{x}\tilde
{u}_{m}\rangle\} \\
& \leq \partial_{t}\tilde{u}_{m}+\frac{m\theta \tilde{u}_{m}}{1+m(T-t)}%
\sup\limits_{v\in\mathcal{V}_0}G(-\sigma_{i}^{\top}\sigma_{i})+\frac{m^{2}\theta^{2}\tilde{u}_{m}|x_{i}%
-a_{i}|^{2}}{(1+m(T-t))^{2}}\sup\limits_{v\in\mathcal{V}_0}G(\sigma_{i}^{\top}\sigma_{i^*})\\
& +\frac{2m\theta \tilde{u}_{m}}{1+m(T-t)}\sup\limits_{v\in\mathcal{V}_0}G((-(x_{i}-a_{i})\beta_{i}^{jk}%
(t,x))_{j,k=1}^{d})+\frac{m\theta \tilde{u}_{m}}{1+m(T-t)}\sup\limits_{v\in\mathcal{V}_0}(a_{i}-x_{i})\alpha_{i}\\
& \leq-\frac{m\theta \tilde{u}_{m}}{1+m(T-t)}|x_{i}-a_{i}|^{2}(\frac
{m}{4(1+m\varepsilon)}-\kappa)\\
& \leq0,
\end{align*}
which implies that $\tilde{u}_{m}$ is a viscosity supersolution of PDE
(\ref{PDE}) if $t\geq T-\varepsilon$. Thus by comparison theorem we obtain for
$(t,x)\in \lbrack T-\varepsilon,T)\times \mathbb{R}^{n}$,
\[
u_{m}(t,x)\leq \tilde{u}_{m}(t,x)\leq(1+m(T-t))^{-\rho}.
\]
The proof is complete.
\end{proof}

Note  that the above result still holds if assumptions (H4)  is  valid only for some $i$. By a similar analysis as in Theorem \ref{th2}, we can show that $c(\{X_{t}^{x_{i};i}=a_{i}\})=0$ for any $t>0$ and $a_{i}\in \mathbb{R}$.  We remark that one can also obtained  this result by Martini's approach and Girsanov's theorem. However, we can also get the convergence rate.
Indeed,
\begin{theorem}
\label{th4}Under the assumptions  \emph{(H1), (H2)} and \emph{(H4)}, we
obtain that for each  $T>0$
\begin{equation}
\mathbb{\hat{E}}[\exp(-\frac{m\theta|X_{T}^{x_{i};i}-a_{i}|^{2}}{2}%
)]\leq(1+m(T\wedge \varepsilon))^{-\rho},\label{e3}%
\end{equation}
where $\theta, \rho $ and $ \varepsilon$ are given in Lemma \emph{\ref{le3}}.
\end{theorem}

By the above result, we can show that the maximal process does not weight a single point.

\begin{corollary}
Assume $d=1$.
Then we have $c(\{B^*_t=a\})=0$ for each $a\in\mathbb{R}$, where  $B^*_t=\sup_{0\leq s\leq t}B_s$.
\end{corollary}
 \begin{proof}
Without loss of generality, assume $t=1$.
For each $m\geq 1$, set $\varphi_m(x)=\exp(-\frac{m^{\frac{2(1+\rho)}{\rho}}\theta|x-a|^{2}}{2})$, where $\theta, \rho$ are given in Lemma \emph{\ref{le3}}.
Then applying Fatou's lemma yields that \[
c(\{B^*_t=a\})\leq \liminf\limits_{m\rightarrow\infty}\mathbb{\hat{E}}[\varphi_m(\sup\{B_{t^m_1}, B_{t^m_2},\cdots, B_{1}\})],
\]
where $t^m_{i}=\frac{i}{m}$ for each $i\leq m$.

By Remark \ref{myq10}  and Theorem \ref{th4}, we conclude that \begin{align*}
\mathbb{\hat{E}}[\varphi_m(\sup\{B_{t^m_1}, B_{t^m_2}\})]\leq& \mathbb{\hat{E}}[\varphi_m(B_{t^m_1}+\sup\{0, B_{t^m_2}-B_{t^m_1}\})] \\
\leq & \mathbb{\hat{E}}[\varphi_m(B_{t^m_1})]+\mathbb{\hat{E}}[\mathbb{\hat{E}}[\varphi_m(y+ B_{t^m_2}-B_{t^m_1})]_{y=B_{t^m_1}}]\\
\leq & 2(1+m^{\frac{2(1+\rho)}{\rho}}{m}^{-1})^{-\rho}\leq \frac{2}{m^{2+\rho}}.
\end{align*}
Iterating the procedure for $m$ times implies that
\[\mathbb{\hat{E}}[\varphi_m(\sup\{B_{t^m_1}, B_{t^m_2},\cdots, B_{1}\})]\leq \frac{1}{m^{1+\rho}}\]
 and this completes the proof.
 \end{proof}
 \begin{example}{\upshape
By Remark \ref{yq1}, we
have $
P(\{X^*_t=y\})=0$,
where $X^*_t$ is the maximal process of $X_t = \int^t_0 \theta_sdB_s$ and this provides a way to study the maxima of non-Markovian It\^{o} process.
Moreover, one can get that $X_t$ has a unique maxima in the interval $[0, t]$. }
\end{example}

Finally, we shall study the capacity of¡¡the  $G$-It\^{o} process staying in a curve.

\begin{theorem}\label{myq1}
Assume \emph{(H1)}, \emph{(H2)} and \emph{(H4)} hold.
Suppose $f$ satisfies $\partial_{x_i}f,\partial_{x_ix_j}^2f \in
C_{b,Lip}(\mathbb{R}^n)$ and there exist two constants $0<\delta\leq
\Delta<\infty$ such that $$\delta\leq
|\sum\limits_{i=1}^n\partial_{x_i}f\sigma_{i}|^2\leq \Delta.$$ Then
for each $T>0$ we have
\begin{equation*}
c(\{f(X^x_T)=0\})=0.
\end{equation*}
\end{theorem}
\begin{proof}
Applying the $G$-It\^{o} formula yields that
\begin{align*}
f(X^x_t)=&f(x)+\int^t_0\partial_{x_i}f\alpha_{i}(s)ds+\int^t_0[\partial_{x_i}f\beta_i^{jk}+\frac{1}{2}\partial^2_{x_ix_l}f\sigma_{ij}\sigma_{lk}](s)d\langle B^j,B^k\rangle_s\\&
+\int^t_0\partial_{x_i}f\sigma_i(s)dB_s.
\end{align*}
Thus $\tilde{X}^x_t=((X^x_t)^{\top},f(X^x_t))^{\top}$ can be seen as  the $G$-It\^{o} process \eqref{SDE1}
corresponding
 to \begin{align*}
\tilde{\alpha}(t)=\left(                 %
  \begin{array}{ccc}
    & {\alpha}(t) \\
    & \partial_{x_i}f\alpha_{i}(t)  \\
  \end{array}
\right),
\tilde{\sigma}(t)=\left(                 %
  \begin{array}{ccc}
    & {\sigma}(t) \\
    & \partial_{x_i}f\sigma_{i}(t)  \\
  \end{array}
\right)\end{align*} and
 \begin{align*}
\tilde{\beta}^{jk}(t)=\left(                 %
  \begin{array}{ccc}
    & {\beta}^{jk}(t) \\
    & [\partial_{x_i}f\beta_i^{jk}+\frac{1}{2}\partial^2_{x_ix_l}f\sigma_{ij}\sigma_{lk}](t)  \\
  \end{array}
\right).
\end{align*}
Thus we have $c(\{f(X^x_T)=0\})=0$ and this completes the proof.
\end{proof}
\begin{example}\label{my1}{\upshape   The property required upon the gradient of the curve $f$ is necessary. Indeed,
we take $n=2$, $d=1$,
$x=0$, $b=0$, $h^{jk}=0$, $\sigma=(1,-1)^T$ and $f(x,y)=x-y$. Then
$f(B_T,B_T)=0, q.s..$ However $\partial_{x}f\sigma_1+\partial_{y}f\sigma_2=0$.
}
\end{example}

\subsection{Some applications}
In this subsection, we shall  identify some non-trivial quasi-continuous Borel measurable functions on $\Omega$ and we always assume  (H1), (H2) and {(H4)} hold.

 \begin{theorem}
\label{le6}Let $\xi \in L_{G}^{1}(\Omega;\mathbb{R}^{k})$ and $A\in
\mathcal{B}(\mathbb{R}^{k})$ with $c(\{ \xi \in \partial A\})=0$. Then
$I_{\{ \xi \in A\}}\in L_{G}^{1}(\Omega)$.
\end{theorem}

\begin{proof}
For each $\epsilon>0$, since $\xi \in L_{G}^{1}(\Omega;\mathbb{R}^{k})$, we can
find an open set $O\subset \Omega$ with $c(O)\leq \frac{\epsilon}{2}$ such that
$\xi|_{O^{c}}$ is continuous. Set $D_{i}=\{x\in$ $\mathbb{R}^{k}:d(x,\partial
A)\leq \frac{1}{i}\}$ and $A_{i}=\{x\in$ $\mathbb{R}^{k}:d(x,\partial
A)<\frac{1}{i}\}$, it is easy to check that $\{ \xi \in D_{i}\} \cap O^{c}$ is
closed, $\{ \xi \in A_{i}\} \subset \{ \xi \in D_{i}\}$ and $\{ \xi \in D_{i}\} \cap
O^{c}\downarrow \{ \xi \in \partial A\} \cap O^{c}$. Then we have%
\[
c(\{ \xi \in D_{i}\} \cap O^{c})\downarrow c(\{ \xi \in \partial A\} \cap O^{c})=0.
\]
Thus we can find an $i_{0}$ such that $c(\{ \xi \in A_{i_{0}}\} \cap O^{c}%
)\leq \frac{\epsilon}{2}$. Set $O_{1}=\{ \xi \in A_{i_{0}}\} \cup O$, it is easy
to verify that $c(O_{1})\leq \epsilon$, $O_{1}^{c}=\{ \xi \in A_{i_{0}}^{c}\} \cap
O^{c}$ is closed and $I_{\{ \xi \in A\}}$ is continuous on $O_{1}^{c}$. Thus
$I_{\{ \xi \in A\}}$ is quasi-continuous, which implies $I_{\{ \xi \in A\}}\in
L_{G}^{1}(\Omega)$.
\end{proof}

Now we consider the capacity of $X_{s}^{t,\xi}$ hitting the boundary of cubes, where $X^{t,\xi}$ is the $G$-It\^{o} process \eqref{SDE1}
starting at  $t$ and  from  the random variable $\xi$.
Then, by the above theorem, we can get a kind of quasi-continuous random variables associated to $G$-It\^{o} processes.

\begin{lemma}\label{myq5}
Let $A=[a,b]$, where $a$, $b\in \mathbb{R}^{n}$ with $a\leq b$. Then for each
given $t\geq0$, $\xi \in L_{G}^{2}(\Omega_{t};\mathbb{R}^{n})$, $s>t$, we have
$c(\{X_{s}^{t,\xi}\in \partial A\})=0$.
\end{lemma}
\begin{proof}
It suffices to prove that $c(\{X_{s}^{t,\xi_i;i}=a_i\})=c(\{X_{s}^{t,\xi_i;i}=b_i\})=0$.
We shall only show that $c(\{X_{s}^{t,\xi_1;1}=a_1\})=0$ and the other cases can be proved in a similar way.
For each $m$, set $\varphi_m(x)=\exp(-\frac{m\theta|x_1-a_1|^{2}}{2})$.
Applying Theorems \ref{th0} and \ref{th4},   we
 conclude that
\[
\mathbb{\hat{E}}[\varphi_{m}(X_{s}^{t,\xi})] \leq(1+m((s-t)\wedge \varepsilon))^{-\rho}
\]
 Letting $m\rightarrow\infty$ yields the desired result and this completes the proof.
\end{proof}
\begin{theorem}
\label{th7}Let $A_{i}=[a^{i},b^{i}]$ with $a^{i}$, $b^{i}\in \mathbb{R}^{n}$, $a^{i}\leq
b^{i}$ for $i\geq1$ and $D\in \mathcal{B}(\mathbb{R}^{n})$ with $\partial
D\subset \cup_{i=1}^{\infty}\partial A_{i}$. Then for each given $t\geq0$,
$\xi \in L_{G}^{2}(\Omega_{t};\mathbb{R}^{n})$, $s>t$, we have $I_{\{X_{s}%
^{t,\xi}\in D\}}\in L_{G}^{1}(\Omega_{s})$. In particular, $I_{\{X_{s}^{x}\in
D\}}\in L_{G}^{1}(\Omega_{s})$.
\end{theorem}
\begin{proof}
This is a direct consequence of Lemma \ref{myq5} and Theorem \ref{le6}.
\end{proof}

In the following, we only consider the capacity of $B_t$ on the sphere. But the method can be applied  to  deal with $X_{s}^{t,\xi}$.

\begin{lemma}\label{myq2}
Let $D$ be a $d$-dimensional  sphere.
Then we have for each $t>0$,
\[
c(\{B_t\in\partial D\})=0.
\]
\end{lemma}
\begin{proof}
Without loss of generality, we assume $D$ is the unit sphere.
Set $\bar{x}=(x_1,\ldots,x_{d-1})$ and
denote functions $$f(\bar{x}):=\sqrt{1-|\bar{x}|^2}I_{\{|\bar{x}|^2\leq 1\}}. $$
For each $\epsilon>0$, there exists  a nonnegative  function $J^{\epsilon}(\bar{x})\in C^{\infty}_0(\mathbb{R}^{d-1})$ such that
\begin{align*}
J^{\epsilon}(\bar{x})=\begin{cases}& 1, \ \text{if $|\bar{x}|\leq 1-2\epsilon$};
\\
&0,\ \text{if $|\bar{x}|\geq 1-\epsilon$}.
\end{cases}
\end{align*}
Then define function $f^\epsilon(x):=x_d-J^{\epsilon}(\bar{x})f(\bar{x}).$ It is easy to check that $J^{\epsilon}(\bar{x})f(\bar{x})\in C^{\infty}_0(\mathbb{R}^{d-1})$.
Moreover, $|\sum\limits_{i=1}^d\partial_{x_i}f^\epsilon(x)e_i|^2=\sum\limits_{i=1}^{d-1}|\partial_{x_i}f^\epsilon(x)|^2+1.$ Then applying Theorem \ref{myq1}, we obtain for each given $t\geq 0$,
\[
c(\{B^d_t-J^{\epsilon}(\tilde{B}_t)f(\tilde{B}_t)=0\})=0,
\]
where $\tilde{B}_t=(B^1_t,\ldots,B^{d-1}_t)$.
Consequently,
\[
c(\{B^d_t-f(\tilde{B}_t)=0\}\cap\{|\tilde{B}_t|^2\leq 1-2\epsilon\})=0.
\]
 Note that $\{B^d_t-f(\tilde{B}_t)=0\}\cap\{|\tilde{B}_t|^2\leq 1-2\epsilon\}$$\uparrow$$\{B^d_t-f(\tilde{B}_t)=0\}\cap \{|\tilde{B}_t|^2< 1\}$, then by taking $\epsilon\downarrow 0$ we get that
\[
c(\{B^d_t-f(\tilde{B}_t)=0\}\cap \{|\tilde{B}_t|^2< 1\})=0.
\]
From Theorem \ref{th4}, we get $c(\{B^d_t=0\})=0.$ Therefore, we deduce that
\[
c(\{B^d_t-f(\tilde{B}_t)=0\})\leq c(\{B^d_t-f(\tilde{B}_t)=0\}\cap \{|\tilde{B}_t|^2< 1\})+c(\{B^d_t=0\})=0.
\]
By a similar analysis, we also get $c(\{B^d_t+f(\tilde{B}_t)=0\})=0$. Thus
\[
c(\{B_t\in\partial D\})\leq c(\{B^d_t-f(\tilde{B}_t)=0\})+c(\{B^d_t+f(\tilde{B}_t)=0\})=0,
\]
which is the desired result.
\end{proof}

The following result is a direct consequence of Theorem \ref{le6}, Lemmas \ref{myq5} and \ref{myq2}.
\begin{theorem}
\label{myq3}
Suppose $A_{i}$ is a $d$-dimensional  sphere or $[a^{i},b^{i}]$ with $a^{i}$, $b^{i}\in \mathbb{R}^{d}$,
$a^{i}\leq b^{i}$ for $i\geq 1$. If   $D$ is in  $\mathcal{B}(\mathbb{R}^{d})$ with
$\partial D\subset \cup_{i=1}^{\infty}\partial A_{i}$ , then $I_{\{B_{t}%
\in  D\}}\in L_{G}^{1}(\Omega_{t})$ for any $t>0$.
\end{theorem}

\begin{example}{\upshape
Assume $d=1$. Given a function $u\in C_{b,Lip}(\mathbb{R})$. Then  for each given $n\in \mathbb{N}$, we take
\begin{equation*}
{h}_{i}^{n}(x)={\bf 1}_{[-n+\frac{i}{n},-n+\frac{i+1}{n})}(x),i=0,\ldots,2n^{2}-1,\
{h}_{2n^{2}}^{n}=1-\sum_{i=0}^{2n^{2}-1}h_{i}^{n}.
\end{equation*}
We denote $u^n(B_t):=\sum\limits_{i=0}^{2n^2}u(-n+\frac{i}{n}){h}_{i}^{n}(B_t)$.
Then by Theorem \ref{myq3} and  a direct calculation, we conclude $u^n(B_t)\in L^1_G(\Omega_{t})$ and
\[
\lim\limits_{n\mapsto\infty}\mathbb{\hat{E}}[|u^n(B_t)-u(B_t)|]=0,
\]
which can be seen  as a counterpart of the  approximation of  function in the nonlinear expectation  theory. In particular, it provides a method to construct the approximation of an  admissible control under the $G$-expectation framework, more details can be founded in \cite{HJ2}.
}
\end{example}

\section{Quasi-continuous processes}
In this section,  we shall study the integrable processes under the $G$-expectation framework.
First, we  consider the characterization of $M^p_G(0,T)$.
Then we apply Krylov's estimates to get  some quasi-continuous processes.

\subsection{Characterization of $M^p_G(0,T)$}
We shall give a characterization of the space $M_{G}^{p}(0,T)$ for each $T>0$
and $p\geq1$, which generalizes the results in \cite{DHP11}.

Set $\mathcal{F}%
_{t}=\mathcal{B}(\Omega_{t})$ for $t\in \lbrack0,T]$ and the distance%
\[
\rho((t,\omega),(t^{\prime},\omega^{\prime}))=|t-t^{\prime}|+\max_{s\in
\lbrack0,T]}|\omega_{s}-\omega_{s}^{\prime}|,\ \text{ for }(t,\omega),(t^{\prime
},\omega^{\prime})\in \lbrack0,T]\times \Omega_{T}.
\]
Define, for each $p\geq1$,%
\[
\mathbb{M}^{p}(0,T)=\{ \eta:\text{progressively measurable on }[0,T]\times
\Omega_{T}\text{ and }\mathbb{\hat{E}}[\int_{0}^{T}|\eta_{t}|^{p}dt]<\infty \}
\]
and the corresponding capacity%
\[
\hat{c}(A)=\frac{1}{T}\mathbb{\hat{E}}[\int_{0}^{T}I_{A}(t,\omega)dt],\text{
for each progressively measurable set }A\subset \lbrack0,T]\times \Omega_{T}.
\]

\begin{proposition}
\label{nee-pro1}Let $A$ be a progressively measurable set in $[0,T]\times
\Omega_{T}$. Then $I_{A}=0$ $\hat{c}$-q.s. if and only if
$\int_{0}^{T}I_{A}(t,\cdot)dt=0$ $c$-q.s..
\end{proposition}

\begin{proof}
It is obvious $\int_{0}^{T}I_{A}(t,\cdot)dt\geq0$. Thus we can
easily get $\mathbb{\hat{E}}[\int_{0}^{T}I_{A}(t,\omega)dt]=0$ if
and only if $c(\{ \int_{0}^{T}I_{A}(t,\cdot)dt>0\})=0$, which
completes the proof.
\end{proof}

In the following, we do not distinguish  the progressively measurable
process $\eta$ from $\eta^{\prime}$ if $\hat{c}(\{ \eta \not
=\eta^{\prime}\})=0$.

\begin{proposition}
For each $p\geq1$, $\mathbb{M}^{p}(0,T)$ is a Banach space under the norm
$||\eta||_{\mathbb{M}^{p}}:=(\mathbb{\hat{E}}[\int_{0}^{T}|\eta_{t}%
|^{p}dt])^{1/p}$.
\end{proposition}

\begin{proof}
The proof is the same as the classical case and we omit it.
\end{proof}

It is clear that $M_{G}^{0}(0,T)\subset \mathbb{M}^{p}(0,T)$ for any $p\geq1$.
Thus $M_{G}^{p}(0,T)$ is a closed subspace of $\mathbb{M}^{p}(0,T)$. Also we
set%
\[
M_{c}(0,T)=\{ \text{all adapted processes }\eta \text{ in }C_{b}([0,T]\times
\Omega_{T})\}.
\]

\begin{proposition}
\label{com} For each $p\geq1$, the completion of $M_{c}(0,T)$ under the norm
$||\cdot||_{\mathbb{M}^{p}}$ is $M_{G}^{p}(0,T)$.
\end{proposition}

\begin{proof}
We first prove that the completion of $M_{c}(0,T)$ under the norm
$||\cdot||_{\mathbb{M}^{p}}$  is included in $M_{G}^{p}(0,T)$. For each fixed
$\eta \in M_{c}(0,T)$, we set $$\eta_{t}^{k}(\cdot)=\sum_{i=0}^{k-1}\eta_{(iT)/k}
(\cdot)I_{[\frac{iT}{k},\frac{(i+1)T}{k})}(t).$$ By Theorem \ref{new-qua}, it is easy to
verify that $\eta^{k}\in M_{G}^{p}(0,T)$. For each $\varepsilon>0$, since
$\mathcal{P}$ is weakly compact, there exists a compact set $K\subset
\Omega_{T}$ such that $\mathbb{\hat{E}}[I_{K^{c}}]\leq \varepsilon$. Thus%
\begin{align*}
\mathbb{\hat{E}}[\int_{0}^{T}|\eta_{t}-\eta_{t}^{k}|^{p}dt]  &  \leq
\mathbb{\hat{E}}[I_{K}\int_{0}^{T}|\eta_{t}-\eta_{t}^{k}|^{p}dt]+\mathbb{\hat
{E}}[I_{K^{c}}\int_{0}^{T}|\eta_{t}-\eta_{t}^{k}|^{p}dt]\\
&  \leq \sup_{(t,\omega)\in \lbrack0,T]\times K}T|\eta_{t}(\omega)-\eta_{t}%
^{k}(\omega)|^{p}+(2l)^{p}T\varepsilon,
\end{align*}
where $l$ is the upper bound of $\eta$. Noting that $[0,T]\times K$ is compact
and $\eta \in C_{b}([0,T]\times \Omega_{T})$, thus
\[
\underset{k\rightarrow \infty}{\limsup}\mathbb{\hat{E}}[\int_{0}^{T}|\eta
_{t}-\eta_{t}^{k}|^{p}dt]\leq(2l)^{p}T\varepsilon.
\]
Since $\varepsilon$ is arbitrary, we get $||\eta^{k}-\eta||_{\mathbb{M}^{p}}\rightarrow0$ as
$k\rightarrow \infty$. Thus $\eta \in M_{G}^{p}(0,T)$, which implies the desired result.

Now we prove the converse part. For each given $\bar{\eta}_{t}=\sum
_{i=0}^{N-1}\xi_{i}I_{[t_{i},t_{i+1})}(t)\in M_{G}^{0}(0,T)$, we can
find $\{ \phi_{k}^{i}:k\geq1\} \subset C([0,\infty))$, $i<N$, $k\geq 1$ so
that $supp(\phi_{k}^{i})\subset (t_{i},t_{i+1})$ and $\int_{0}^{T}|\phi_{k}^{i}(t)-I_{[t_{i},t_{i+1})}(t)|^{p}dt\rightarrow0$
as $k\rightarrow \infty$. Set $\bar{\eta}_{t}^{k}=\sum_{i=0}^{N-1}\xi_{i}%
\phi_{k}^{i}(t)$, it is easy to check that $\bar{\eta}^{k}\in M_{c}(0,T)$ and
$||\bar{\eta}^{k}-\bar{\eta}||_{\mathbb{M}^{p}}\rightarrow0$ as $k\rightarrow
\infty$. Thus each element of
$M_{G}^{p}(0,T)$ belongs to the completion of $M_{c}(0,T)$ under
the norm $||\cdot||_{\mathbb{M}^{p}}$, which completes the proof.
\end{proof}

\begin{definition}
A progressively measurable process $\eta:[0,T]\times \Omega_{T}\rightarrow
\mathbb{R}$ is called quasi-continuous (q.c.), if for each $\varepsilon>0$,
there exists a progressively measurable  open set $G$ in $[0,T]\times \Omega_{T}$ such that
$\hat{c}(G)<\varepsilon$ and $\eta|_{G^{c}}$ is continuous.
\end{definition}

\begin{remark}{\upshape
Our definition of quasi-continuous  process is different from
the ones in \cite{Song11, Song12}.}
\end{remark}

\begin{definition}
We say that a progressively measurable process $\eta:[0,T]\times \Omega
_{T}\rightarrow \mathbb{R}$ has a quasi-continuous version if there exists a
quasi-continuous process $\eta^{\prime}$ such that $\hat{c}(\{ \eta \not =%
\eta^{\prime}\})=0$.
\end{definition}

\begin{theorem}
\label{quasi-c}For each $p\geq1$,%
\begin{align*}
M_{G}^{p}(0,T)=\{\eta \in \mathbb{M}^{p}(0,T)\ :\ \
&\lim_{N\rightarrow
\infty}\mathbb{\hat{E}}[\int_{0}^{T}|\eta_{t}|^{p}I_{\{|\eta_{t}|\geq
N\}}dt]=0 \
 \text{and}\\ & \eta \text{ has a quasi-continuous
version}\}.
 \end{align*}
\end{theorem}

\begin{proof}
We denote%
\begin{align*}
J_{p}=\{\eta \in \mathbb{M}^{p}(0,T)\ :\ \ &\lim_{N\rightarrow
\infty}\mathbb{\hat{E}}[\int_{0}^{T}|\eta_{t}|^{p}I_{\{|\eta_{t}|\geq
N\}}dt]=0 \
 \text{and}\\ & \eta \text{ has a quasi-continuous
version}\}.
 \end{align*}
Noting that the completion of $M_{c}(0,T)$ under the norm $||\cdot
||_{\mathbb{M}^{p}}$ is $M_{G}^{p}(0,T)$, then, by the same analysis
as in Propositions 18 and 24 in \cite{DHP11}, we can get
$M_{G}^{p}(0,T)\subset J_{p}$.

On the other hand, for each $\eta \in J_{p}$, we need to prove that $\eta \in
M_{G}^{p}(0,T)$. Without loss of generality, we assume that $\eta$ is quasi-continuous. For each
$N>0$, set $\eta^{N}=(\eta \wedge N)\vee(-N)$, since $\mathbb{\hat{E}}[\int
_{0}^{T}|\eta_{t}-\eta_{t}^{N}|^{p}dt]\leq \mathbb{\hat{E}}[\int_{0}^{T}%
|\eta_{t}|^{p}I_{\{|\eta_{t}|\geq N\}}dt]\rightarrow0$ as $N\rightarrow \infty
$, it suffices to show that $\eta^{N}\in M_{G}^{p}(0,T)$ for each fixed
$N>0$. For each $\varepsilon>0$, there exists a compact set $K_{\varepsilon
}\subset \Omega_{T}$ such that $\mathbb{\hat{E}}[I_{K_{\varepsilon}^{c}}%
]\leq \varepsilon$ and a progressively measurable open set $G_{\varepsilon}\subset
\lbrack0,T]\times \Omega_{T}$ such that $\hat{c}(G_{\varepsilon})<\varepsilon$ and $\eta^{N}|_{G_{\varepsilon}^{c}}$ is continuous. By Tietze's extension
theorem, there exists a function $\tilde{\eta}^{N,\varepsilon}\in C_{b}([0,T]\times
\Omega_{T})$ such that $|\tilde{\eta}^{N,\varepsilon}|\leq N$ and $\tilde
{\eta}^{N,\varepsilon}|_{G_{\varepsilon}^{c}}=\eta^{N}|_{G_{\varepsilon}^{c}}%
$. For each $k\geq1$, we set $F^{i,k}=G_{\varepsilon}^{c}\cap([t_{i}^{k},t_{i+1}^{k}]\times \Omega_{T})$
for $i\leq k-1$, where $t_{i}^{k}=\frac{iT}{k}$ for $i=0,\ldots,k$. Since $G_{\varepsilon}^{c}$ is progressively measurable, we can get
$F^{i,k}\in \mathcal{B}([0,t_{i+1}^{k}])\times \mathcal{B}(\Omega_{t_{i+1}^{k}%
})$. Since $F^{i,k}$ is closed,  again  by Tietze's extension
theorem, there exists a function $\zeta^{N,i,k}\in C_{b}([0,t_{i+1}^{k}]\times
\Omega_{T})$ such that $\zeta^{N,i,k}\in \mathcal{B}([0,t_{i+1}^{k}%
])\times \mathcal{B}(\Omega_{t_{i+1}^{k}})$, $|\zeta^{N,i,k}|\leq N$
and $\zeta^{N,i,k}|_{F^{i,k}}=\eta^{N}|_{F^{i,k}}$. We denote
$\tilde{\eta}_{t}^{N,k}(\omega)=\sum_{i=0}^{k-1}\zeta^{N,i,k}(t,\omega
)I_{[t_{i}^{k},t_{i+1}^{k})}(t)$ and
\[
\bar{\eta}_{t}^{N,k}(\omega)=\tilde{\eta
}^{N,k}(t-\frac{T}{k},\omega)I_{[t_{1}^{k},T)}(t),\ \bar{\eta}_{t}%
^{N,\varepsilon,k}(\omega)=\tilde{\eta}^{N,\varepsilon}(t-\frac{T}{k}%
,\omega)I_{[t_{1}^{k},T)}(t).
\]
A similar analysis as in Proposition \ref{com} implies that $\bar{\eta
}^{N,k}\in M_{G}^{p}(0,T)$. Moreover, we obtain that
\begin{align*}
&  \mathbb{\hat{E}}[\int_{0}^{T}|\eta_{t}^{N}-\bar{\eta}_{t}^{N,k}|^{p}dt]\\
&  \leq3^{p-1}(\mathbb{\hat{E}}[\int_{0}^{T}|\eta_{t}^{N}-\tilde{\eta}%
_{t}^{N,\varepsilon}|^{p}dt]+\mathbb{\hat{E}}[\int_{0}^{T}|\tilde{\eta}%
_{t}^{N,\varepsilon}-\bar{\eta}_{t}^{N,\varepsilon,k}|^{p}dt]+\mathbb{\hat{E}%
}[\int_{0}^{T}|\bar{\eta}_{t}^{N,\varepsilon,k}-\bar{\eta}_{t}^{N,k}%
|^{p}dt])\\
&  \leq3^{p-1}(\mathbb{\hat{E}}[\int_{0}^{T}|\eta_{t}^{N}-\tilde{\eta}%
_{t}^{N,\varepsilon}|^{p}dt]+\mathbb{\hat{E}}[\int_{0}^{T}|\tilde{\eta}%
_{t}^{N,\varepsilon}-\bar{\eta}_{t}^{N,\varepsilon,k}|^{p}dt]+\mathbb{\hat{E}%
}[\int_{0}^{T}|\tilde{\eta}_{t}^{N,\varepsilon}-\tilde{\eta}_{t}^{N,k}%
|^{p}dt])\\
&
\leq3^{p-1}(2(2N)^{p}T\varepsilon+\mathbb{\hat{E}}[\int_{0}^{T}|\tilde
{\eta}_{t}^{N,\varepsilon}-\bar{\eta}_{t}^{N,\varepsilon,k}|^{p}dt])\\
&  \leq3^{p-1}(2(2N)^{p}T\varepsilon+(2N)^{p}\frac{T}{k}+\mathbb{\hat{E}}%
[\int_{t_{1}^{k}}^{T}|\tilde{\eta}_{t}^{N,\varepsilon}-\bar{\eta}%
_{t}^{N,\varepsilon,k}|^{p}dt])\\
&  \leq3^{p-1}(2(2N)^{p}T\varepsilon+(2N)^{p}\frac{T}{k}+\mathbb{\hat{E}%
}[I_{K_{\varepsilon}^{c}}\int_{t_{1}^{k}}^{T}|\tilde{\eta}_{t}^{N,\varepsilon
}-\bar{\eta}_{t}^{N,\varepsilon,k}|^{p}dt]  +\mathbb{\hat{E}}[I_{K_{\varepsilon}}\int_{t_{1}^{k}}^{T}|\tilde{\eta}%
_{t}^{N,\varepsilon}-\bar{\eta}_{t}^{N,\varepsilon,k}|^{p}dt])\\
&
\leq3^{p-1}(3(2N)^{p}T\varepsilon+(2N)^{p}\frac{T}{k}+\sup_{(t,\omega
)\in \lbrack t_{1}^{k},T]\times
K_{\varepsilon}}T|\tilde{\eta}^{N,\varepsilon
}(t,\omega)-\tilde{\eta}^{N,\varepsilon}(t-\frac{T}{k},\omega)|^{p}).
\end{align*}
Noting that $[0,T]\times K_{\varepsilon}$ is compact and $\tilde{\eta
}^{N,\varepsilon}\in C_{b}([0,T]\times \Omega_{T})$, thus
\[
\underset{k\rightarrow \infty}{\limsup}\mathbb{\hat{E}}[\int_{0}^{T}|\eta
_{t}^{N}-\bar{\eta}_{t}^{N,k}|^{p}dt]\leq(6N)^{p}T\varepsilon,
\]
which implies $\eta^{N}\in M_{G}^{p}(0,T)$. The proof is complete.
\end{proof}

\begin{remark}{
\upshape
Note that the Tietze's extension theorem
cannot ensure the  extension of a progressively measurable process is also progressively measurable. Then we provide an
alternative way to prove
 the characterization of $M_{G}^{p}(0,T)$, which is  different from that of \cite{DHP11}.
}
\end{remark}

By Theorem
\ref{quasi-c}, we immediately have the following result.
\begin{corollary}
\label{new-cor1}Let $\eta \in M_{G}^{1}(0,T)$ and $f\in C_{b}%
([0,T]\times \mathbb{R})$. Then $(f(t,\eta_{t}))_{t\leq T}\in
M_{G}^{p}(0,T)$ for any $p\geq1$.
\end{corollary}

\begin{theorem}
\label{nee-the1}Let $\eta^{k}$ be in $M_{G}^{1}(0,T)$, $k\geq1$, such
that
$\eta^{k}\downarrow \eta$ $\hat{c}$-q.s.. Then $\mathbb{\hat{E}}[\int_{0}%
^{T}\eta_{t}^{k}dt]\downarrow
\mathbb{\hat{E}}[\int_{0}^{T}\eta_{t}dt]$.
Moreover, if $\eta \in M_{G}^{1}(0,T)$, then $\mathbb{\hat{E}}[\int_{0}%
^{T}|\eta_{t}^{k}-\eta_{t}|dt]\downarrow0$.
\end{theorem}

\begin{proof}
Since $\eta^{k}\in M_{G}^{1}(0,T)$, we can choose $\eta^{k,N}\in M_{G}%
^{0}(0,T)$ such that
$\mathbb{\hat{E}}[\int_{0}^{T}|\eta_{t}^{k}-\eta
_{t}^{k,N}|dt]\rightarrow0$ as $N\rightarrow \infty$. It is easy to
check that
$\int_{0}^{T}\eta_{t}^{k,N}dt\in L_{G}^{1}(\Omega_{T})$ and $\mathbb{\hat{E}%
}[|\int_{0}^{T}\eta_{t}^{k,N}dt-\int_{0}^{T}\eta_{t}^{k}dt|]\leq
\mathbb{\hat{E}}[\int_{0}^{T}|\eta_{t}^{k}-\eta_{t}^{k,N}|dt]$. Then we get
$\int _{0}^{T}\eta_{t}^{k}dt\in L_{G}^{1}(\Omega_{T})$ for $k\geq1$.
By Proposition \ref{nee-pro1} and Theorem \ref{quasi-c}, it is easy
to verify that $\int _{0}^{T}\eta_{t}^{k}dt\downarrow
\int_{0}^{T}\eta_{t}dt$ $c$-q.s.. Thus, applying
Theorem \ref{new-qua1} yields that $\mathbb{\hat{E}}[\int_{0}^{T}\eta_{t}%
^{k}dt]\downarrow \mathbb{\hat{E}}[\int_{0}^{T}\eta_{t}dt]$. If
$\eta \in
M_{G}^{1}(0,T)$, then $|\eta^{k}-\eta|\in M_{G}^{1}(0,T)$ and $|\eta^{k}%
-\eta|\downarrow0$ $\hat{c}$-q.s., which implies that $\mathbb{\hat{E}}[\int_{0}^{T}%
|\eta_{t}^{k}-\eta_{t}|dt]\downarrow0$.
\end{proof}

The following example shows that $M_{G}^{p}(0,T)$ is strictly
contained in $\mathbb{M}^{p}(0,T)$.
\begin{example}{\upshape
Suppose $0<\underline{\sigma}^{2}<\bar{\sigma}^{2}<\infty$, $T>0$.
We consider $1$-dimensional $G$-Brownian motion $(B_{t})_{t\geq0}$.
$(\langle B\rangle _{t})_{t\geq0}$ is the quadratic process of
$(B_{t})_{t\geq0}$. Let
\[
\eta_{t}=I_{\{ \langle
B\rangle_{t}=\frac{(\underline{\sigma}^{2}+\bar{\sigma
}^{2})t}{2}\}}\text{ for }t\leq T.
\]
Then we claim that $\eta \not \in M_{G}^{1}(0,T)$. Indeed we can
choose
$f^{k}(t,x)\in C_{b}([0,T]\times \mathbb{R})$, $k\geq1$, such that%
\[
f^{k}(t,x)=1\text{ for }|x-\frac{(\underline{\sigma}^{2}+\bar{\sigma}^{2}%
)t}{2}|\leq \frac{T}{k};f^{k}(t,x)=0\text{ for
}|x-\frac{(\underline{\sigma }^{2}+\bar{\sigma}^{2})t}{2}|\geq
\frac{2T}{k}.
\]
Set $g^{k}=\wedge_{i=1}^{k}f^{i}$, it is easy to check that
$g^{k}\in C_{b}([0,T]\times \mathbb{R})$, $g^{k}(t,x)=1$ for
$|x-\frac{(\underline {\sigma}^{2}+\bar{\sigma}^{2})t}{2}|\leq
\frac{T}{k}$ and $g^{k}\downarrow
I_{\{x=\frac{(\underline{\sigma}^{2}+\bar{\sigma}^{2})t}{2}\}}$.
Since $g^{k}(t,\langle B\rangle_{t})\downarrow \eta_{t}$,  we have $g^{k}(t,\langle B\rangle_{t})\in
M_{G}^{1}(0,T)$ by
Corollary \ref{new-cor1}. If $\eta \in M_{G}^{1}(0,T)$, then  it  following  from Theorem
\ref{nee-the1} that $\mathbb{\hat
{E}}[\int_{0}^{T}|g^{k}(t,\langle
B\rangle_{t})-\eta_{t}|dt]\downarrow0$. On the other hand, by the
representation of $\mathbb{\hat{E}}[\cdot]$ in \cite{DHP11}, there
exists a probability measure $P\in \mathcal{P}$ such that
$\langle B\rangle_{t}=((\frac{(\underline{\sigma}^{2}+\bar{\sigma}^{2})}%
{2}-\frac{1}{k})\vee \underline{\sigma}^{2})t$ $P$-a.s.. Therefore we have $\mathbb{\hat{E}%
}[\int_{0}^{T}|g^{k}(t,\langle B\rangle_{t})-\eta_{t}|dt]\geq
E_{P}[\int _{0}^{T}|g^{k}(t,\langle B\rangle_{t})-\eta_{t}|dt]=T$
and this contradiction implies that $\eta \not \in M_{G}^{1}(0,T)$.}
\end{example}

\subsection{$G$-integrable processes}
In the above subsection, we give the characterization of $M_{G}^{p}(0,T)$.
However, it is also difficult to check that a  progressively measurable process
is quasi-continuous. Then the present section is devoted to finding some Borel measurable
functions on $[0,T]\times\Omega_T$ are quasi-continuous processes.

In this section, we always assume $n\leq d$ and (H1)-(H3) hold.
For some fixed $x_0\in\mathbb{R}^n$,  consider the $G$-It\^{o}
process $X^{x_0}$ given by \eqref{SDE1}.
For convenience, we set $X=X^{x_0}$.
\begin{theorem}[Krylov's estimates]\label{my4}
For each $\delta>0$ and $p\geq n$,
there exists a constant $N$ depending on  $p,{\lambda},{\Lambda}, L,G$ and $\delta$ such that for each Borel measurable function
$f(t,x)$ and $g(x)$,
\begin{align*}
&\mathbb{\hat{E}}[\int^{\infty}_0 \exp(-\delta t)|f(t,X_t)|dt]\leq N\|f\|_{L^{p+1}([0,\infty)\times\mathbb{R}^n)},\\
&\mathbb{\hat{E}}[\int^{\infty}_0 \exp(-\delta t)|g(X_t)|dt]\leq N\|g\|_{L^p(\mathbb{R}^n)}.
\end{align*}
\end{theorem}
\begin{proof}
Let $\mathcal{P}$ be the weakly compact set that represents
$\mathbb{\hat{E}}$. By Corollary 5.7 in Chapter 3 of \cite{P10}, we
obtain that $d\langle B^j, B^k\rangle_t = \hat{\gamma}^{jk}_tdt$
q.s. and $\underline{\sigma}^2tI_{d\times d}\leq
\hat{\gamma}_t=(\hat{\gamma}^{jk}_t)_{j,k=1}^d\leq\bar{\sigma}^2tI_{d\times
d} $. Note that $B$ is a martingale  on the probability  space
$(\Omega, (\mathcal{F}_t)_{t\geq 0}, P)$ for each $P\in\mathcal{P}.$
Then it is easy to check that \[
W^P_t:=\int^t_0\hat{\gamma}_s^{-\frac{1}{2}}dB_s, \ \ P-a.s.
\]
is a Brownian motion  on $(\Omega, (\mathcal{F}_t)_{t\geq 0}, P)$. Thus  we have
\begin{equation*}
X_{t}=x_0+\int_{0}^{t}\alpha_sds+\int_{0}^{t}\beta^{jk}_s\hat{\gamma}^{jk}_sds+\int_{0}^{t}\sigma_s\hat{\gamma}_s^{\frac{1}{2}}dW^P_s, \ P-a.s..
\end{equation*}
Applying Theorem 3.4 in Chapter 2 of Krylov \cite{KV1} (see also \cite{KV}), we can find a constant ${N}$ depending on  $p, {\lambda}, \Lambda, L, G$  and $\delta$   such that for each Borel measurable function
$f(t,x)$,
\begin{align*}
{E}_P[\int^{\infty}_0 \exp(-\delta t)|f(t,X_t)|dt]\leq \bar{N}\|f\|_{L^{p+1}([0,T]\times \mathbb{R}^n)}.
\end{align*}
Therefore, we have
\begin{align*}
\mathbb{\hat{E}}[\int^{\infty}_0 \exp(-\delta t)|f(t,X_t)|dt]=\sup\limits_{P\in\mathcal{P}}E_P[\int^{\infty}_0 \exp(-\delta t)|f(t,X_t)|dt]\leq {N}\|f\|_{L^{p+1}([0,T]\times \mathbb{R}^n)}
\end{align*}
and the second inequality can be proved in a similar way.
\end{proof}

The following estimates are from Theorem \ref{my4}.
\begin{corollary}\label{my5}
For each $T>0$ and $p\geq n$, there exists a constant $N_T$ depending on  $p,{\lambda},{\Lambda}, L,G$ and $T$ such that for each Borel measurable function
$f(t,x)$ and $g(x)$,
\begin{align*}
&\mathbb{\hat{E}}[\int^{T}_0 |f(t,X_t)|dt]\leq N_T\|f\|_{L^{p+1}([0,T]\times\mathbb{R}^n)},\\
&\mathbb{\hat{E}}[\int^{T}_0 |g(X_t)|dt]\leq N_T\|g\|_{L^p(\mathbb{R}^n)}.
\end{align*}
\end{corollary}

From
 now on, we shall use Krylov's estimates to generate some
quasi-continuous processes.
\begin{lemma}\label{myq8}
\begin{description}
\item[(i)] If $\psi$ is in $L^{p}([0,T]\times \mathbb{R}^{n})$ with $p\geq
n+1$, then for each $T>0$, we have $(\psi(t,X_{t}))_{t\leq T}\in M_{G}%
^{1}(0,T)$. Moreover, for each $\psi^{\prime}=\psi,a.e.$, we have
$\psi ^{\prime}(\cdot,X_{\cdot})=\psi(\cdot,X_{\cdot})$;

\item[(ii)] If $\varphi$ is in $L^{p}(\mathbb{R}^{n})$ with $p\geq n$, then
for each $T>0$, we have $(\varphi(X_{t}))_{t\leq T}\in
M_{G}^{1}(0,T)$. Moreover, for each
$\varphi^{\prime}=\varphi,a.e.$, we have $\varphi^{\prime
}(X_{\cdot})=\varphi(X_{\cdot})$.
\end{description}
\end{lemma}
\begin{proof} We only prove (ii), since (i) can be proved in a similar way. Note that there exists a
sequence of bounded continuous  functions $(\varphi^{k})_{k\geq 1}$, which  converges  to $\varphi$ in ${L^{p}(\mathbb{R}^{n})}$.
Then by Corollary \ref{my5}, we can find a constant $C^{\prime}$
so that
\[
\lim \limits_{k\rightarrow
\infty}\mathbb{\hat{E}}[\int_{0}^{T}|\varphi
^{k}-\varphi|(X_{t})dt]\leq C^{\prime}\lim \limits_{k\rightarrow \infty}%
\Vert \varphi^{k}-\varphi \Vert_{L^{p}(\mathbb{R}^{n})}=0.
\]
By Theorem \ref{quasi-c}, we can get $(\varphi^{k}(X_{t}))_{t\leq
T}\in M_{G}^{1}(0,T)$ for each $k\geq1$. Thus we obtain
$(\varphi(X_{t}))_{t\leq T}\in M_{G}^{1}(0,T)$.

Assume $\varphi=\varphi^{\prime},a.e.$. Applying  Corollary \ref{my5} again, we conclude that
\[
\mathbb{\hat{E}}[\int_{0}^{T}|\varphi^{\prime}-\varphi|(X_{t})dt]\leq
C^{\prime}\Vert \varphi^{\prime}-\varphi
\Vert_{L^{p}(\mathbb{R}^{n})}=0,
\]
which completes the proof.
\end{proof}

\begin{theorem}
\label{myq13}  Let
$(\varphi^{k})_{k\geq1}$ be a sequence of $\mathbb{R}^{n}$-valued Borel
measurable functions and $|\varphi ^{k}(x)|\leq \bar{C}(1+|x|^{l})$,
$k\geq1$ for some constants $\bar{C}$ and $l$. If
$\varphi^{k}\rightarrow \varphi$, a.e., then for each $T>0$ and
$p\geq 1$,
\[
\lim \limits_{k\rightarrow
\infty}\mathbb{\hat{E}}[\int_{0}^{T}|\varphi
^{k}(X_{t})-\varphi(X_{t})|^{p}dt]=0.
\]

\end{theorem}

\begin{proof}
By Lemma \ref{myq8}, we may assume that $|\varphi(x)|\leq
\bar{C}(1+|x|^{l})$. For each fixed $N>0$, we have
\begin{align*}
\hat{\mathbb{E}}[\int_{0}^{T}|\varphi^{k}(X_{t})-\varphi(X_{t})|^{p}dt]\leq
&
\hat{\mathbb{E}}[\int_{0}^{T}|\varphi^{k}(X_{t})-\varphi(X_{t})|^{p}%
I_{\{|X_{t}|\leq N\}}dt]\\
&  +\hat{\mathbb{E}}[\int_{0}^{T}|\varphi^{k}(X_{t})-\varphi(X_{t}%
)|^{p}I_{\{|X_{t}|\geq N\}}dt].
\end{align*}
By Corollary \ref{my5}, there exists a constant $C^{\prime}$
independent of $k$ such that
\[
\hat{\mathbb{E}}[\int_{0}^{T}|\varphi^{k}(X_{t})-\varphi(X_{t})|^{p}%
I_{\{|X_{t}|\leq N\}}dt]\leq C^{\prime}|\int_{\{|x|\leq N\}}|\varphi
^{k}(x)-\varphi(x)|^{np}dx|^{\frac{1}{n}}.
\]
Then applying Lesbesgue's dominated convergence theorem yields that%
\[
\hat{\mathbb{E}}[\int_{0}^{T}|\varphi^{k}(X_{t})-\varphi(X_{t})|^{p}%
I_{\{|X_{t}|\leq N\}}dt]\rightarrow0\text{ as }k\rightarrow \infty.
\]
Noting that $|\varphi^{k}(X_{t})-\varphi(X_{t})|^{p}I_{\{|X_{t}|\geq N\}}%
\leq \frac{(2\bar{C})^{p}}{N}(1+|X_{t}|^{l})^{p}|X_{t}|$, then we
get
\[
\underset{k\rightarrow \infty}{\limsup}\hat{\mathbb{E}}[\int_{0}^{T}%
|\varphi^{k}(X_{t})-\varphi(X_{t})|^{p}dt]\leq
\frac{(2\bar{C})^{p}}{N}\int
_{0}^{T}\hat{\mathbb{E}}[(1+|X_{t}|^{l})^{p}|X_{t}|]dt.
\]
Since $N$ can be arbitrarily large, we obtain
\[
\lim \limits_{k\rightarrow
\infty}\hat{\mathbb{E}}[\int_{0}^{T}|\varphi
^{k}(X_{t})-\varphi(X_{t})|^{p}dt]=0,
\]
which is the desired result.
\end{proof}

Theorem \ref{myq13}  can be seen as a weak dominated convergence
theorem for the $G$-It\^{o} processes. By this result, we obtain
\begin{theorem}
If $\varphi$ is a
$\mathbb{R}^{n}$-valued Borel measurable function of polynomial growth,
then  we have $(\varphi(X_{t}))_{t\leq T}\in
M_{G}^{2}(0,T)$ for each $T>0$.
\end{theorem}

\begin{proof}
We can find a sequence of continuous functions $(\varphi^{k})_{k\geq1}$ with compact support, such that $\varphi^{k}$ converges to
$\varphi$ a.e. and $|\varphi^{k}(x)|\leq \bar{C}(1+|x|^{l})$, where
$\bar{C}$, $l$ are constants independent of $k$. Then by Theorem
\ref{myq13}, for each $T>0$, we conclude that
\[
\lim \limits_{k\rightarrow
\infty}\mathbb{\hat{E}}[\int_{0}^{T}|\varphi
^{k}-\varphi|^{2}(X_{t})dt]=0.
\]
Since $(\varphi^{k}(X_{t}))_{t\leq T}\in M_{G}^{2}(0,T)$ for each
$k$ by
Theorem \ref{quasi-c}, we derive that $(\varphi(X_{t}))_{t\leq T}\in M_{G}%
^{2}(0,T)$ and this completes the proof.
\end{proof}

\textbf{Acknowledgement}: The authors would like to thank Prof. Shige Peng for   his helpful discussions and
suggestions. We also thank the anonymous reviewer for carefully reading the manuscript and
giving many valuable suggestions.

\end{document}